\documentclass[ejs,preprint]{imsart}
\usepackage[english]{babel}
\usepackage{amsmath,amsthm,amsfonts,graphicx}
\usepackage[T1]{fontenc}
\usepackage{fullpage}
\usepackage{here}
\usepackage{hyperref}

\newtheorem{thm}{Theorem}[section]
\newtheorem{lem}{Lemma}[section]
\newtheorem{Def}{Definition}[section]
\newtheorem{prop}{Proposition}[section]
\newtheorem{Ass}{Assumption}
\renewcommand{\c}{\mathbf{c}}
\newcommand{\m}{\mathbf{m}}
\def\Var{\mathop{\rm Var}}

\begin{document}

\begin{frontmatter}

\title{Fast rates for empirical vector quantization}
\runtitle{Fast rates for quantization}

\begin{aug}

\author{Clément Levrard\ead[label=e1]{clement.levrard@u-psud.fr}}
\address{Université Paris Sud 11 and Paris 6,\\
\printead{e1}}

\runauthor{C.Levrard}
\end{aug}

\begin{abstract}
       We consider the rate of convergence of the expected loss of empirically optimal vector quantizers. Earlier results show that the mean-squared expected distortion for any fixed distribution supported on a bounded set and satisfying some regularity conditions decreases at the rate $\mathcal{O}(\log n/n)$. We prove that this rate is actually $\mathcal{O}(1/n)$. Although these conditions are hard to check, we show that well-polarized distributions with continuous densities supported on a bounded set are included in the scope of this result. 
         
\end{abstract}


         \end{frontmatter}

         \section{Introduction}
         
          Clustering is the problem of identifying groupings of similar points that are relatively far one from each others, or, in other words, to partition the data into dissimilar groups of similar items. For a comprehensive introduction to this topic, the reader is referred to the monograph of Graf and Luschgy \cite{GL00}. Isolate meaningful groups from a cloud of data is a topic of interest in many fields, from social science to biology. In fact this issue originates in the theory of signal processing in the late 40's, known as the quantization issue, or lossy data compression (see Gersho and Gray \cite{Gersho92} for a comprehensive approach of this topic). More precisely, let $X_1, \hdots, X_n$ denote $n$ random variables, independent and identically distributed, drawn from a distribution $P$ over $\mathbb{R}^d$, equipped with its Euclidean norm $\|.\|$, and let $Q$ denote a $k$-quantizer, that is a map from $\mathbb{R}^d$ to $\mathbb{R}^d$ such that $ \mbox{Card}(Q(\mathbb{R}^d)) \leq k$. Let $\mathbf{c} \in (\mathbb{R}^d)^k$ be a concatenation of $k$ $d$-dimensional vectors $c_1, \hdots, c_k$. Without loss of generality we only consider quantizers of the type $x\longmapsto c_i$, where $\| x- c_i \| = \min_{j=1, \hdots, k}{\|x - c_j\|}$. The $c_i$'s are called clusters. To measure how well the quantizer $Q$ performs in representing the source distribution, a possible way is to look at
          \[
          R(\mathbf{c}) = \mathbb{E}\|X - Q(X)\|^2 = \mathbb{E} \min_{j=1, \hdots, k}{\|X-c_j\|^2},
          \]
          when $\mathbb{E}X^2 < \infty$. 
         The goal here is to find a set of clusters $\hat{\mathbf{c}}_n$, drawn from the data $X_1, \hdots, X_n$, whose distortion is as close as possible to the optimal distortion $R^* = \inf_{\mathbf{c} \in (\mathbb{R}^d)^k}{R(\mathbf{c})}$. To solve the problem, most approaches to date attempt to implement the principle of empirical error minimization in the vector quantization context. According to this principle, good clusters can be found by searching for ones that minimize the empirical distortion over the training data, defined by 
         \[
         \hat{R}_n(\mathbf{c}) = \frac{1}{n}\sum_{i=1}^{n}{(X_i - Q(X_i))^2} = \frac{1}{n}\sum_{i=1}^{n}{ \min_{j=1, \hdots, k}{\|X_i - c_j\|^2}}.
         \]
          The existence of such empirically optimal clusters has been established by Graf an Luschgy \cite[Theorem 4.12]{GL00}. Let us denote by $\hat{\mathbf{c}}_n$ one of these vectors of empirically optimal clusters. If the training data represents the source well, $\hat{\mathbf{c}}_n$ will hopefully perform near optimally also on the real source. Roughly, this means that we expect $R(\hat{\mathbf{c}}_n) \approx R^*$. The problem of quantifying how good empirically designed clusters are, compared to the truly optimal ones, has been extensively studied, see for instance Linder \cite{Linder02}.
         
         To reach the later goal, a standard route is to exploit the Wasserstein distance between the empirical distribution and the source distribution, to derive upper bounds on the average distortion of empirically optimal clusters. Following this approach, Pollard \cite{Pollard81} proved that if $\mathbb{E}\|X\|^2 < \infty$, then $R(\hat{\mathbf{c}}_n) - R^* \longrightarrow 0$ almost surely, as $n \rightarrow \infty$.  More recently, Linder, Lugosi and Zeger  \cite{Linder94}, and Biau, Devroye and Lugosi \cite{Biau08} showed that if the support of $P$ is bounded, then $\mathbb{E} \left (R(\hat{\mathbf{c}}_n) - R^* \right ) = \mathcal{O}(1/\sqrt{n})$, using techniques borrowed from statistical learning theory. Bartlett, Linder and Lugosi \cite{Bartlett98} established that this rate is minimax over distributions supported on a finite set of points.
         
         However, faster rates can be achieved, using methods inspired from statistical learning theory. For example, it is shown by Chou \cite{Chou94}, following a result of Pollard \cite{Pollard82}, that $R(\hat{\mathbf{c}}_n) - R^* = \mathcal{O}_{\mathbb{P}}\left ( 1/n \right )$, under some regularity conditions on the source distribution. Nevertheless, this consistency result does not provide any information on how many training samples are needed to ensure that the average distortion of empirically optimal clusters is close to the optimum. Antos, Györfi and György established in \cite{Antos04} that $\mathbb{E} (R(\hat{\mathbf{c}}_n) - R^*) =\mathcal{O}(\log n / n)$ under the same conditions, paying a $ \log n $ factor to derive a non-asymptotic bound. It is worth pointing out that the conditions cannot be checked in practice, and consequently remain of theoretical nature. Moreover, the rate of $1/n$ for the average distortion can be achieved when the source distribution is supported on a finite set of points. Consequently, an open question is to know wether this optimal rate can be attained for more general distributions.
         
         In the present paper, we improve previous results of Antos, György and Györfi \cite{Antos04}, by getting rid of the $\log n$ factor. Besides, we express Pollard's condition in a more reader-friendly framework, involving the density of the source distribution. To this aim we use statistical learning arguments and prove that the average distortion of empirically optimal clusters decreases at the rate $\mathcal{O}(1 /n)$. To get this result we use techniques such as the localization principle borrowed from Massart, Blanchard and Bousquet \cite{Blanchard08} or Koltchinskii \cite{Koltchinskii04}. The condition we offer can be easily interpreted as margin-type condition, similar to the ones of Massart and Nedelec in \cite{Massart06}, showing a clear connection between statistical learning theory and vector quantization.
         
         The paper is organized as follows. In Section 2 we introduce notation and definitions of interest. In Section 3 we offer our main results. These results are discussed in Section 4, and illustrated on examples such as Gaussian mixtures or quasi-finite distribution. Finally, proofs are gathered in Section 5.
         
         \section{The quantization problem}
         
         Throughout the paper, $X_1, \hdots, X_n$ is a sequence of independent $\mathbb{R}^d$-valued random observations  with the same distribution $P$ as a generic random variable X. To frame the quantization problem as a statistical learning one, we first have to consider quantization as a contrast minimization issue. To this aim we introduce the following notation.
         Let $\mathbf{c}= (c_1, \hdots, c_n)$ be the set of possible clusters. The contrast function $\gamma$ is defined as

         \begin{center}
                 $ \gamma : \left \{
                 \mbox{\begin{tabular}{ccc}
                 $\left (\mathbb{R}^d \right )^k \times \mathbb{R}^d$& $\longrightarrow$ & $\mathbb{R}$ \\
                 \hspace{0.45cm} $ (\mathbf{c},x$) &$ \longmapsto $& $\underset{j=1, \hdots, k}{\min}{\left \| x-c_j \right \|^2}$
                 \end{tabular}}
                 \right . .$
                 \end{center}
                 
                Within this framework, the risk $R(\mathbf{c})$ takes the form $R (Q) = R(\mathbf{c}) = P \gamma (\mathbf{c},.)$, where $Pf(.)$ means inegration of the function $f$ with respect to $P$. In the same way, if $P_n$ denotes the empirical distribution that is induced on $\mathbb{R}^d$ by the $n$-sample $X_1, \hdots, X_n$, we can express the empirical risk $\hat{R}_n(Q)$ as $P_n \gamma( \mathbf{c}, .)$.
                 
                 Note that, within this context, an optimal $\mathbf{c}^*$ minimizes $P \gamma(\mathbf{c},.)$, whereas $\hat{\mathbf{c}}_n \in \arg\min_{\mathbf{c} \in (\mathbb{R}^d)^k}{P_n \gamma( \mathbf{c},.)}$. It is worth pointing out that the existence of both $\mathbf{c}$ and $\mathbf{c}^*$ are guaranteed by Graf and Luschgy \cite[Theorem 4.12]{GL00}. In the sequel we denote by $\mathcal{M}$ the set of such minimizers of the true risk $P \gamma (\mathbf{c},.)$, so that $\mathbf{c}^* \in \mathcal{M}$.  To measure how well a vector of clusters $\mathbf{c}$ performs compared to an optimal one, we will make use of the loss  
                 \[
                 \ell(\mathbf{c},\mathbf{c}^*) = R(\mathbf{c}) - R(\mathbf{c}^*) = P \left (\gamma ( \mathbf{c}, . ) - \gamma (\mathbf{c}^*,.) \right ).
                 \]                 
                 Troughout the paper we will use the following assumptions on the source distribution. Let $\mathcal{B}(0,M)$ denote the closed ball of radius $M$, with $M \geq 0$.

                  \begin{Ass}[\textbf{Peak Power Constraint}]
                  
                  The distribution $P$ is such that  $P({\mathcal{B}}(0,1)) = 1$,
                  
                  \end{Ass}

                   Note that Assumption 1 is stronger than the requirement $\mathbb{E} \left \|X^2\right \| < \infty $, as it imposes a $L_\infty$-boundedness condition on the random variable $X$. For conveniency we assume that the distribution is bounded by $1$. However, it is important to note that our results hold for random variables $X$ bounded from above by an arbitrary $M$. We will also need the following regularity requirement, first introduced by Pollard \cite{Pollard82}.
                
                  \bigskip
                  
                  \begin{Ass}[\textbf{Pollard's regularity condition}]
                  
                  The distribution $P$ satisfies the following two conditions:
                  
                  \begin{enumerate}
                  \item $P$ has a continuous density $f$ with respect to Lebesgue measure on $\mathbb{R}^d$,
                  \item The Hessian matrix of $ \mathbf{c} \longmapsto P \gamma (\mathbf{c},.) $ is positive definite for all optimal vector of clusters $\mathbf{c}^*$.
                  \end{enumerate}
                  
                  \end{Ass}

                  One can point out that Condition 1 of Assumption 2 does not guarantee the existence of a second derivative for the expectation of the contrast function. Nevertheless Assumption 1 and Condition 1 of Assumption 2 are enough to guarantee that the map $ \mathbf{c} \longmapsto P \gamma (\mathbf{c},.) $ is twice differentiable. Let $V_i$ be the Voronoi cell associated with $c_i$, for $i=1, \hdots, k$. In this situation, the Hessian matrix is composed of the following $d \times d$ blocks:
                   \[
                   H(\mathbf{c})_{i,j} = 
                   \left \{
                   \mbox{\begin{tabular}{ccc}
                   $ 2 P(V_i) - 2 \sum_{\ell \neq i} r_{i\ell}^{-1} \sigma \left [ f(x)(x-c_i)(x-c_i)^t \mathbf{1}_{\partial(V_i \cap V_\ell)} \right ]$  &for& $i=j $\\
                    $-2{r_{ij}}^{-1} \sigma \left [ f(x)(x-c_i)(x-c_j)^t \mathbf{1}_{\partial(V_i \cap V_j)} \right ] $& for  &$i \neq j$
                   \end{tabular}}
                   \right . ,
                   \]
                   where $r_{ij} = \left \| c_i - c_j \right \|$, $\partial(V_i \cap V_j)$ denotes the possibly empty common face of $V_i$ and $V_j$, and $\sigma$ means integration with respect to the $(d-1)$-dimensional Lebesgue measure. For a proof of that statement, we refer to Pollard \cite{Pollard82}.

                   When Assumption 1 and Assumption 2 are satisfied, Chou  \cite{Chou94} proved that $ \ell(\hat{\mathbf{c}}_n,\mathbf{c}^*) = \mathcal{O}_{\mathbb{P}}(1/n) $, whereas Antos, Györfi, and György established that $\mathbb{E} \ell(\hat{\mathbf{c}}_n,\mathbf{c}^*) \leq C \frac{\log n}{n}$, where $C$ is a constant depending on the distribution $P$.
                   
                   The proof of these two results are both based on arguments which have a connection with the localization principle (\cite{Massart03},  \cite{Koltchinskii04}), which provides faster rates of convergence when the expectation and the variance of $\gamma(\c,.) - \gamma(\c^*,.)$ are connected. To prove his result, Pollard used conditions  under which the distortion and the Euclidean distance are connected, and used chaining arguments to bound from above a term which looks like a Rademacher complexity, constrained on an area around an optimal vector of clusters. Note that Koltchinskii \cite{Koltchinskii04} used a similar method to apply the localization principle. On the other hand, Antos, Györfi and György exploited Pollard's condition, and used a concentration inequality based on the fact that the variance and the expectation of the distortion are connected to get their result. Interestingly, this point of view has been developped by Blanchard, Bousquet and Massart \cite{Blanchard08} to get bounds on the classification risk of the SVM, using the localization principle. That is the approach that will be followed in the present document.

                  \section{Main results}

                 We are now in a position to state our main result.

                  \begin{thm}
                  
                  Assume that Assumption 1 and Assumption 2 are satisfied. Then, denoting by $\hat{\mathbf{c}}_n$ an empirical risk minimizer, we have
                  
                  \[
                  \mathbb{E} \ell(\hat{\mathbf{c}}_n,\mathbf{c}^*)  \leq \frac{C_0}{n},
                  \]
                  where $C_0$ is a positive constant depending on $P$, $k$ and $d$.
                  \end{thm}

                  This result improves previous non-asymptotic results of Antos, Györfi and György \cite{Antos04}, Linder, Lugosi and Zeger \cite{Linder94}, showing that a convergence rate of $1/n$ can be achieved in expectation. 
                   To prove Theorem 3.1, the key result is based on a version of Talagrand's inequality due to Bousquet \cite{Bousquet02} and its application to localization, following the approach of Massart and Nedelec \cite{Massart06}. The main point is to connect $\Var \left ( \gamma (\mathbf{c},.) - \gamma(\mathbf{c}^*,.) \right )$ to $P \left ( \gamma (\mathbf{c},.) - \gamma(\mathbf{c}^*,.) \right )$ for all possible $\mathbf{c}$. To be more precise, Pollard's condition involves differentiability of the distortion, therefore $ \gamma (\mathbf{c},.) - \gamma(\mathbf{c}^*,.)$ is naturally linked to $\|\mathbf{c}-\mathbf{c}^*\|$, the Euclidean distance between $\mathbf{c}$ and $\mathbf{c}^*$. However, it is noteworthy that, mimicing the proof of Antos, Györfi and György \cite[Corollary 1]{Antos04}, we have in fact:
               
                   \bigskip                                 
                 
                 \begin{prop}
                 
                 Suppose that Assumption 1 and Assumption 2 are satisfied. Then there exists two positive constants $A_1$ and $A_2$, depending on the distribution $P$, such that 
                 
                 \begin{enumerate}
                 \item (H1) : $ \forall \mathbf{c} \in \mathcal{B}(0,1)^k \quad \|\mathbf{c}-\mathbf{c}^*(\c)\|^2 \leq A_1 \ell(\mathbf{c},\mathbf{c}^*(\c))$,
                 \item (H2) : $\forall \mathbf{c} \in \mathcal{B}(0,1)^k \quad \forall \c^* \in \mathcal{M} \quad \Var( \gamma(\mathbf{c},.) - \gamma(\mathbf{c}^*,.) ) \leq A_2 \| \c - \c^*\|^2$,  
                 \end{enumerate}
                 
                 where $\c^*(\c) \in \underset{\c^* \in \mathcal{M}}{\arg\min}{\|\c - \c^*\|}$.
                 \end{prop}
                 
                 When considering several possible optimal vector of clusters, we have to choose one to be compared with our empirical vector $\hat{\c}_n$. A nearest optimal vector of clusters $\c^*(\hat{\c}_n)$ is a natural choice. It is important to note that, for every $\c \in (\mathbb{R}^d)^k$ and $\c^* \in \mathcal{M}$, $\ell(\c, \c^*(\c)) = \ell(\c,\c^*)$. Consequently, Theorem 3.1 holds for every possible $\c^* \in \mathcal{M}$. 
                 Besides it is easy to see, using the compacity of $\mathcal{{B}} (0,1)$, that there is only a finite set of optimal clusters $\c^*$ when Assumption 1 is satisfied and the Hessian matrixes $H(\mathbf{c}^*)$ are definite positive for every possible $\mathbf{c}^*$. This compacity argument is also the key to turn respectively the local positiveness of $H(\mathbf{c}^*)$ into property (H1) and the regularity of the contrast function $\gamma $ into the global property (H2). These two properties are exactly matching the two parts of the proof of Antos, Györfi and György \cite[Corollary 1]{Antos04}, which in turn implies Proposition 3.1. Note also that, from Proposition 3.1 we get $ \Var( \gamma(\mathbf{c},.) - \gamma(\mathbf{c}^*(\c),.) ) \leq A_1 A_2 \ell(\mathbf{c},\mathbf{c}^*(\c))$. This allows us to use localization techniques such as in the paper of Blanchard, Bousquet and Massart \cite{Blanchard08}.
                                  
                 Pollard's regularity condition (Assumption 2) involves second derivatives of the distortion. Consequently, checking Assumption 2, even theoretically, remains a hard issue. We give a more general condition regarding the $L_\infty$-norm of the density $f$ on the boundaries of Voronoi diagram, for the distribution to satisfy Assumption 2. We recall that $\mathcal{M}$ denotes the set of all possible optimal clusters $\c^*$. 
                 
                 \begin{thm}
                 
                   Denote by $V_i^*$ the Voronoi cell associated with $c_i^*$ in the Voronoi diagram associated with $\mathbf{c}^*$, by $N^*$ the union of all possible boundaries of Voronoi cells with respect to all possible optimal vector of clusters $\mathbf{c}^*$, and by $\Gamma$ the Gamma function. Let 
$B = \underset{\mathbf{c}^* \in \mathcal{M}, i \neq j}{\inf}{\|c^*_i - c^*_j\|}$. Suppose that
 \begin{center}
 { $\|f_{\left | N^* \right .}\|_{\infty} \leq \frac{\Gamma\left ( \frac{d}{2} \right ) B }{2^{d+5}  \pi^{d/2}} \underset{\c^* \in \mathcal{M},i=1, \hdots ,k}{\inf} {P (V_i^*)}$}.
 \end{center} 
Then $P$ satisties Assumption 2.
                 \end{thm}
                 
                  The proof is given in Section 5. It is important to note that, for general distributions supported on $\mathcal{B}(0,M)$, we can state a similar theorem, involving $M^{d+1}$ in the right-hand side of the inequality in Theorem 3.2. However, a source distribution supported on $\mathcal{B}(0,M)$ can be turned into a distribution supported on $\mathcal{B}(0,1)$, using an homothetic transformation. Therefore we will only state results for a distribution supported on $\mathcal{B}(0,1)$.

                 This theorem emphasizes the idea that if $P$ is well concentrated around its optimal clusters, then some localization conditions can hold and therefore it is a favorable case. The intuition behind this result is given by the extremal case where Voronoi cells boundaries are empty with respect to $P$. This case is described in detail in Section 4. Moreover, the notion of a well-concentrated distribution looks like margin-type conditions for the classification case, as described by Massart and Nedelec \cite{Massart06}. This confirms the intuition of an easy-to-quantize distribution, when the poles are well-separated.

                 \section{Discussion and examples}
                 
                 \subsection{Minimax lower bound}
                             
                             Let $\mathcal{P}$ denote the set of probability distributions on $\mathcal{B}(0,1)$. Bartlett, Linder and Lugosi \cite{Bartlett98} offered a minimax lower bound for general distributions:
                             \[
                             \underset{P \in \mathcal{P}}{\sup}{\mathbb{E}\ell(\hat{\c}_n, \c^*) \geq c_0 \sqrt{\frac{k_n^{1-4/d}}{n}}}.
                             \]
                              Consequently, for general distributions, this minimax bound mathches the upper bound on $\mathbb{E}\ell(\hat{\c}_n,\c^*)$ Linder, Lugosi, and Zeger \cite{Linder94} obtained. A question of interest is to know whether the rate of $1/n$ we get in Theorem 3.1 is minimax over the set of distributions which satisfies Assumption 1 and Assumption 2. Proposition 4.1 below answers this question.
                             
                             \begin{prop}
                             Let $n$ be an integer, and denote by $\mathcal{D}$ the set of distributions $P$ satisfying Assumption 1 and Assumption 2. Then there exists a constant $c_0$ and an integer $k_n$ such that, for any $k_n$-point quantizer $\c_n$ 
                             
                             \[
                             \underset{P \in \mathcal{D}}{\sup}{\mathbb{E}\ell(\hat{\c}_n, \c^*) \geq c_0 \sqrt{\frac{k_n^{1-4/d}}{n}}}.
                             \]
                             \end{prop}
                             There is no contradiction between Theorem 3.1 and Proposition 4.1. In fact, in Theorem 3.1, $k$ is fixed, whereas, in Proposition 4.1, $k_n$ strongly depends on $n$. Therefore, it is an interesting point to know whether we can get such a minimax bound when $k$ is fixed. 
                             
                             The proof of Proposition 4.1 follows the proof of Bartlett, Linder and Lugosi \cite[Theorem 1]{Bartlett98}, and it is therefore omitted in this paper. The main idea is to replace the distribution supported on $2n$ points proposed by these authors in Step 3 of the proof, with a distribution supported on $2n$ small balls satisfying Assumption 1 and Assumption 2.

                             \subsection{ Assumption 1 is necessary}

                             The original result of Pollard \cite{Pollard82} assume only that $\mathbb{E} \|X\|^2 < \infty$, to get an asymptotic rate of $ \mathcal{O}_{\mathbb{P}} \left(1/n \right) $. Consequently, it is an interesting question to know whether Assumption 1 can be replaced with the assumption $\mathbb{E}\|X\|^2 < \infty$ in Theorem 3.1.
                             In fact, Assumption 1 is useful to get a global localization result from a local one, through a compacity argument. This is precisely the result of Proposition 3.1, which provides us with the global argument required for applying some localization result from a local regularity condition. However, following the idea of Antos, Györfi and György \cite{Antos04}, it is possible to suppose only that $\mathbb{E}\|X\|^2 < \infty$ and nevertheless get (H2) in Proposition 3.1, as expressed in the following result.

                             \begin{prop}
                             Suppose that $\mathbb{E}\|X\|^2<\infty $ and that the set of all possible optimal clusters $\c^*$ is finite. Then there exists a constant $A_2$, depending on $P$, such that

                             \[
                             \forall \mathbf{c} \neq \mathbf{c}^* \quad \Var(\gamma(\mathbf{c},.) - \gamma(\mathbf{c}^*,.)) \leq A_2 \|\mathbf{c}-\mathbf{c}^*\|^2.
                             \]
                             
                             \end{prop}
                                                   
                             A proof of Proposition 4.2 can be directly deduced from the proof of \cite[Theorem 2]{Antos04}. Consequently it is omitted in this paper. According to Proposition 4.2, we can expect to control the variance of our process indexed by $\mathbf{c}$ and $\c^*$ with the Euclidean distance $\| \c - \c^*\|$, even if the support of $P$ is not contained within a ball. Unfortunately, when the distribution is not supported on a bounded set, there are cases where the term $\ell(\mathbf{c},\mathbf{c}^*(\c))$ cannot dominate $\|\mathbf{c}-\mathbf{c}^*(\c)\|^2$ for all $\mathbf{c}$, as expressed in the following counter-example.
                             Let $\eta >0$,  $q(\eta) = \frac{e^{\eta/2 + 2R -1}}{3}$, and  define the density $f$ of the distribution $P$ supported on $\mathbb{R}$ by
 
                             \begin{center}
                             \[
\label{contre exemple}
                             f(x) = 
                             \left \{
                             \mbox{\begin{tabular}{ccl}
                              $\frac{1}{3 \eta} $& if &$ x \in [0,\eta]$ \\
                              $\frac{1}{3 \eta}$ & if & $x \in [R,R + \eta]$ \\
                              $q(\eta) e^{-x}$ & if & $x > 2 R -1 + \frac{\eta}{2}$\\
                              $0$ & \multicolumn{2}{c}{elsewhere} 
                             \end{tabular}}.
                             \right .
                             \]
                             \end{center}
                             
                             \begin{prop}
                             Set $\eta =2$, $R=10$, and define $\mathbf{c}_n = (0,n,n^2)$. Then
                             \begin{itemize}
                             \item[$(i)$] $P$ satisfies Assumption 2.
                             \item[$(ii)$] We have $ \ell(\mathbf{c}_n,\mathbf{c}^*(\mathbf{c}_n)){\longrightarrow}{ P\|x\|^2 < \infty} $ as $ n \longrightarrow \infty$.
                             \item [$(iii)$] We have $\|\mathbf{c}_n - \mathbf{c}^*(\mathbf{c}_n)\|^2 {\sim} n^4$ as $ n \longrightarrow \infty$. 
                             \end{itemize}
                             \end{prop}

                             One easily deduces from Proposition 4.3 that the distribution $P$ satisfies Assumption 2, but fails to satisfy $(H1)$ in Proposition 3.1. Therefore Assumption 1 is necessary to get the result of Theorem 3.2. The intuiton behind this counter-example is that two phenomenons prevent $\ell(\mathbf{c},\mathbf{c}^*(\c))$ from being at most proportional to $\| \mathbf{c} - \mathbf{c}^*(\c)\|^2$ when $\mathbf{c}$ is arbitrarily far from $0$. Firstly, the underlying measure "erase" the Euclidean distance in the expression of $\ell(\mathbf{c},\mathbf{c}^*(\c))$, which implies that $\ell(\mathbf{c}_n,\mathbf{c}^*(\mathbf{c}_n))$ converges. Therefore, a suitable criterion to link $\ell(\mathbf{c},\mathbf{c}^*(\c))$ and $\Var(\gamma(\mathbf{c},.) - \gamma(\mathbf{c}^*`(\c),.))$ should probably involve a weight drawn from the tail of $P$. Typically we expect such a criterion to be a function of $\|\mathbf{c}-\mathbf{c}^*(\c)\|^2$, taking into account a tail constraint on the distribution $P$ we consider.
 Secondly, this example shows that, if for instance we take the $3$-quantizer $\c_n=(n,n^2,n^4)$, the relative loss $\ell(\mathbf{c}_n,\mathbf{c}^*(\mathbf{c}_n))$ will mostly depend on the contribution of the smallest cluster $n$, when $n$ grows to infinity, whereas $\|\mathbf{c}_n - \mathbf{c}^*(\mathbf{c}_n)\|^2$ essentially depends on the distance to the most far from $0$ cluster $n^4$.

                   To conclude, the Euclidean distance does not take into account the weight induced by the underlying distribution over the space. Thus, when Assumption 1 is released, dominant clusters for the Euclidean distance from $\mathbf{c}^*$ are essentially the most far ones. On the other hand, when integrating with respect to $P$, far-from-zero clusters loose their influence in the loss $\ell(\mathbf{c}_n,\mathbf{c}^*(\mathbf{c}_n))$. 
                            
                 \subsection{A toy example}
                 
                 In this subsection we intend to understand which conditions on the density $f$ can guarantee that the Hessian matrixes $H$ are positive. To this aim we consider an extremal case, in which the probability distribution is supported on small balls scattered in $\mathcal{B}(0,1)$. Roughly, if the balls are small enough and far one from each others, the optimal quantization points should be the center of these balls.
                 These are the ideas which are behind the following proposition.
                 
                 \begin{prop}
                 
                 Let $z_1, \hdots, z_k$ be vectors in $\mathbb{R}^d$. Let $\rho$ be a positive number and $R = \underset{i\neq j}{\inf}\quad {\| z_i - z_j\|}$ be the smallest possible distance between these vectors. Let the distribution $P$ be defined as follows                 
                 \[
                 \forall i \in \{ 1, \hdots, k \} \quad
                 \left \{
                 \begin{aligned}
                 & P \left ( \mathcal{B}(z_i, \rho) \right )= \frac{1}{k} \\
                 & P_{ \left| \mathcal{B} (z_i, \rho) \right. } \sim \mathcal{U}_{\left | \mathcal{B}(z_i, \rho) \right .}
                 \end{aligned}
                 \right.  ,
                 \]
                 where $\mathcal{U}_{\left | \mathcal{B}(z_i, \rho) \right .}$ denotes the uniform distribution over $\mathcal{B}(z_i, \rho)$.
                 Then, if $\left ( \frac{R}{2} - 3 \rho \right )^2 \geq \frac{2 \rho^2 d}{d+2}$, the optimal $k$-centroid vector is $(z_1, \hdots, z_k)$.

                 \end{prop}

                 The proof of Proposition 4.4, which is given in Section 5, is inspired from a proof of Bartlett, Linder and Lugosi \cite[Step 3]{Bartlett98}. It is interesting to note that Proposition 4.4 can be extended to the situation where we assume that the underlying distribution is supported on $k$  small enough subsets. In this context, if each subset has a not too small $P$-measure,  and if those subsets are far enough one from each others, it can be proved in the same way that an optimal quantizer has a point in every small subset.

                 Let us now consider the distribution described in Proposition 4.3, with relevant values for $\rho$ and $R$. We immediatly see that if $R/2 > \rho$, then every boundary of the Voronoi diagram for the optimal vector of clusters lies in a null-measured area. Thus, for this distribution,

                 \begin{center}
                   $H(\mathbf{c}^*) = 
\begin{pmatrix}
  \frac{1}{k} I_d & \cdots & 0 \\
  \vdots & \ddots & \vdots \\
  0 & \cdots & \frac{1}{k} I_d
  \end{pmatrix} $, 
                   \end{center}
  which is clearly positive.

                 This short example illustrates the idea behind Theorem 3.2. Namely, if the density of the distribution is not too big at the boundaries of the Voronoi diagram associated with every optimal $k$-quantizer, then the Hessian matrix $H$ will roughly behave as a positive diagonal matrix. Thus Pollard's condition (Assumption 2) will be satisfied. This most favorable case is in fact derived from the special case where the distribution is supported on $k$ points. Antos, Györfi and György \cite{Antos04} proved that if the distribution has only a finite number of atoms, then the expected distortion $\ell(\hat{\mathbf{c}}_n,\mathbf{c}^*)$ is at most $ C/n $, where $C$ is a constant. Here we spread the atoms into small balls to give a density to the distribution and match regularity conditions.

                 \subsection{Quasi-Gaussian mixture example}
                 
                The aim of this subsection is to apply our results to the Gaussian mixtures in dimension $d=2$. However, since the distribution support of a Gaussian random variable is not bounded, we will restrict ourselves to the "quasi-Gaussian mixture" model, which is defined as follows.

                 Let the density $f$ of the distribution $P$ be defined by 
                 \[
                 f(x) = \sum_{i=1}^{k}{\frac{p_i}{N_i} \frac{1}{2 \pi \sigma^2} e^{-\frac{\|x-m_i\|^2}{2 \sigma^2}}}\mathbf{1}_{\mathcal{B}(0,1)},
                 \]
 where $N_i$ denotes a normalization constant for each Gaussian variable. To ensure this model to be close to the Gaussian mixture model, we assume that there exists a constant $\varepsilon \in \left [0, 1\right ]$ such that, for $i=1, \hdots, k$, $N_i \geq 1 - \varepsilon$. Denote by $\tilde{B} = {\inf_{i \neq j}}{\|m_i - m_j\|}$ the smallest possible distance between two different means of the mixture. To avoid boundary issues we suppose that, for all $i = 1, \hdots, k$, $\mathcal{B}(m_i, \tilde{B}/3) \subset \mathcal{B}(0,1)$. For such a model, we have:

                 \begin{prop}
                 Suppose that
                 \[
                 \frac{p_{min}}{p_{max}} \geq \max {\left(\frac{288 k \sigma^2}{(1-\varepsilon) \tilde{B}^2(1 - e^{-\tilde{B}^2/{288\sigma^2}})},  \frac{96 k}{(1 - \varepsilon)\sigma^2 \tilde{B}(e^{\tilde{B}/{72\sigma^2}}-1)}\right ) }.
                 \]
                 Then $P$ satisfies Assumption 2.
                 \end{prop}

                  The inequality we propose as a condition in Proposition 4.5 can be decomposed as follows. If 
                  \[
                  \frac{p_{min}}{p_{max}} \geq \frac{288 k \sigma^2}{(1-\varepsilon) \tilde{B}^2(1 - e^{-\tilde{B}^2/{288\sigma^2}})},
                  \]
                   then the optimal vector of clusters $\c^*$ is close to the vector of means of the mixture $\mathbf{m} = (m_1, \hdots, m_k)$. Knowing that, we can locate the Voronoï boundaries of the Voronoï diagram associated to $\c^*$ and apply Theorem 3.2. This leads to the second term of the maximum in Proposition 4.5.
                  
                  This condition can be interpreted as a condition on the polarization of the mixture. A favorable case for vector quantization seems to be when the poles of the mixtures are well-separated, which is equivalent to $\sigma$ is small compared to $\tilde{B}$ when considering Gaussian mixtures. Proposition 4.5 just explained how $\sigma$ has to be small compared to $\tilde{B}$, in order to satisfy Assumption 2 and therefore apply Theorem 3.1, to reach an improved convergence rate of $1/n$ for the loss $\ell(\hat{\c}_n,\c^*)$. Notice that Proposition 4.5 can be considered as an extension of Proposition 4.4. In these two propositions a key point is to locate $\c^*$, which is possible when the distribution $P$ is well-polarized. The definition of a well-polarized distribution takes two similar forms when looking at Proposition 4.4 or Proposition 4.5. In Proposition 4.4 the favorable case is when the poles are far one from each other, separated by an empty area with respect to $P$, which ensures that the Hessian matrixes $H(\c^*)$ are positive definite (in this case they are diagonal matrixes). When slightly disturbing the framework of Proposition 4.4, it is quite natural to think that the Hessian matrixes $H(\c^*)$ should remain positive definite. Proposition 4.5 is an illustration of this idea: the empty separation area between poles is replaced with an area where the density $f$ is small compared to its value around the poles. The condition on $\sigma$ and $\tilde{B}$ we offer in Proposition 4.5 gives a theoretical definition of a well-polarized distribution for quasi-Gaussian mixtures.
                  
                  It is important to note that our result holds when $k$ is known and match exactly the number of components of the mixture. When the number of cluster $k$ is larger than the number of components $\tilde{k}$ of the mixture, we have no general idea of where the optimal clusters can be placed. Moreover, suppose that we are able to locate the optimal vector of clusters $\c^*$. As explained in the proof of Proposition 4.5, the quantity involved in Proposition 4.5 is in fact $B = \inf_{i\neq j}\|c^*_i - c^*_j\|$. Thus, in this case, we expect $B$ to be much smaller than $\tilde{B}$. Consequently, a condition like in Proposition 4.5 could not involve the natural parameter of the mixture $\tilde{B}$. 
                 
                 The two assumptions $N_i \geq 1 - \varepsilon$ and $\mathcal{B}(m_i, \tilde{B}/3) \subset \mathcal{B}(0,1)$ can easily be satisfied when $P$ is constructed via an homothetic transformation. To see this, take a generic Gaussian mixture on $\mathbb{R}^2$, denote by $\bar{m}_i, i=1, \hdots, k$, its means and by $\bar{\sigma}^2$ its variance. For a given $\varepsilon > 0$, choose $M > 0$ such that, for all $i=1, \hdots, k$, $\int_{\mathcal{B}(0,M)}{e^{-\|x-m_i\|^2/2\sigma^2}dx} \geq 2 \pi \sigma^2 (1-\varepsilon)$ and $\mathcal{B}(m_i, \tilde{B}/3) \subset \mathcal{B}(0,M)$. Denote by $P_0$ the "quasi-Gaussian mixture" we obtain on $ \mathcal{B}(0,M)$ for such an $M$. Then, applying an homothetic transformation with coefficient $1/M$ to $P_0$ provides a quasi-Gaussian mixture on $\mathcal{B}(0,1)$, with means $m_i = \bar{m_i}/M, i=1, \hdots, k$ and variance $\sigma^2 = \bar{\sigma}^2/M^2$. This distribution satisfies both $N_i \geq 1 - \varepsilon$ and $\mathcal{B}(m_i, \tilde{B}/3) \subset \mathcal{B}(0,1)$.

                  \section{Proofs }
                  \subsection{Proof of Theorem 3.1}
                  The proof strongly relies on the localization principle and its application by Blanchard, Bousquet and Massart \cite{Blanchard08}. We start with the following definition.

                  \begin{Def}
                  
                  Let $\Phi$ be a real-valued function. $\Phi$ is called a sub-$\alpha$ function if and only if $\Phi$ is non-decreasing and if the map $x \mapsto \Phi(x)/x^{\alpha}$ is non-increasing.
                  \end{Def}
                  The next theorem is an adaptation of the result of Blanchard, Bousquet and Massart \cite[Theorem 6.1]{Blanchard08}. For the sake of clarity its proof is given in Subsection 5.2.
                  
                  \begin{thm}
                 
               Let $\mathcal{F}$ be a class of bounded measurable functions such that
               \begin{itemize}
               \item[$(i)$] $\forall f \in \mathcal{F} \quad \left\|f\right\|_{\infty} \leq b$,
               \item[$(ii)$] $\forall f \in \mathcal{F} \quad  \Var(f) \leq \omega(f)$.
               \end{itemize}
               Let $K$ be a positive constant, $\Phi$ a sub-$\alpha$ function, $\alpha \in \left[1/2,1\right[$. Then there exists a constant $C(\alpha)$ such that, if $D$ is a constant satisfying $D \leq 6K C(\alpha)$ and $r^*$ is the unique solution of the equation $\Phi(r) = r/D$, the following holds.
                Assume that
               \[
               \forall r \geq r^*, \qquad \mathbb{E} \left ( \sup_{\omega(f)\leq r}{(P-P_n)f} \right ) \leq \Phi(r).
               \]
               Then, for all $x>0$, with probability larger than $1-e^{-x}$,
               \[
               \forall f \in \mathcal{F}, \quad Pf - P_n f \leq K^{-1} \left ( \omega(f) + \left (\frac{6K C(\alpha)}{D} \right )^{\frac{1}{1-\alpha}}r^* + \frac{(9K^2+16Kb)x}{4n} \right ).
               \]
               
               \end{thm}

                This theorem emphasizes the fact that if we are able to control the variance and the complexity term controlled by the variance, we can get a possibly interesting oracle inequality. Obviously the main point is to find a suitable control function for the variance of the process. Here the interesting set is
               \[
               \mathcal{F} = \left \{ \gamma(\mathbf{c},.) - \gamma(\mathbf{c}^*,.), \c \in \mathcal{B}(0,1)^k,\c^* \in\mathcal{M}  \right \}.
               \]
                              According to Section 3 the relevant control function for the variance of the process $\gamma(\mathbf{c},.) - \gamma(\mathbf{c}^*,.)$ is proportional to $\|\mathbf{c}-\mathbf{c}^*\|^2$. Thus it remains to bound from above the quantity
               \[
               \mathbb{E} \left ( \sup_{\c^* \in \mathcal{M}, \|\mathbf{c}-\mathbf{c}^*\|^2 \leq \delta} {(P_n-P)(\gamma(\mathbf{c}^*,.) - \gamma(\mathbf{c},.))}\right ).
               \]
               This is done in the following proposition.
               
               \begin{prop}
               Suppose that $P$ satisfies Assumption 1. Then
               
               \[
               \mathbb{E} \left ( \sup_{\c^* \in \mathcal{M}, \|\mathbf{c}-\mathbf{c}^*\|^2 \leq \delta} {(P_n-P)(\gamma(\mathbf{c}^*,.) - \gamma(\mathbf{c},.))}\right ) \leq \sqrt{\delta} \frac{C}{n},
               \]
               where $C$ is a constant depending on $k$, $d$, and $P$.
               \end{prop}

               Assuming that Assumption 1 and Assumption 2 are satisfied, we can apply Theorem 5.1, with $w(\mathbf{c},\c^*) = A_2 \|\mathbf{c} - \mathbf{c}^* \|^2$ and $b=2$.

                             \begin{lem}
                             Let $D>0$. For all $\c^* \in \mathcal{M}$, $x>0$ and $K>D/7$, if $r^*$ is the (unique) solution of $\Phi(\delta) = \delta/D$, then we have, with probability larger than $1-e^{-x}$,
                             \[
                             (P-P_n)(\gamma(\mathbf{c},.) - \gamma(\mathbf{c}^*,.)) \leq K^{-1} A_2 \|\mathbf{c}-\mathbf{c}^*\|^2 + \frac{50K}{D^2} r^* + \frac{K+18}{n}x.
                             \]
                             \end{lem}
                             
                             We are now in a position to prove Theorem 3.1. Take $\c^* = \c^*(\c)$, a nearest optimal vector of clusters to $\c$, and use $(H2)$ to connect $\|\mathbf{c}-\mathbf{c}^*(\c)\|^2$ to $\ell(\mathbf{c},\mathbf{c}^*(\c))$. Introducing the explicit form $r^* = \frac{C^2 D^2}{n}$, we get, with $K=2A_1 A_2$, $D=6K$, and probability larger than $1-e^{-x}$
                             \[
                             1/2 (P-P_n)(\gamma(\hat{\mathbf{c}}_n,.) - \gamma(\mathbf{c}^*(\c),.)) \leq \frac{50 C^2 D^2}{36 K n} + \frac{K + 18}{n}x.
                             \]
                             Observing that $P_n(\gamma(\hat{\mathbf{c}}_n,.) - \gamma(\mathbf{c}^*(\c),.)) \leq 0$, and taking expectation leads to, for all $\c^* \in \mathcal{M}$,
                             \[
                             \mathbb{E} \ell(\hat{\mathbf{c}}_n,\mathbf{c}^*) \leq \frac{C_0}{n},
                             \]
                             for some constant $C_0>0$ depending only on $k$, $d$, and $P$.

                             \subsection{Proof of Theorem 5.1}                   
         
         This proof is a modification of the proof of  Blanchard, Bousquet and Massart \cite[Theorem 6.1]{Blanchard08}. For $r\geq0$, set 
         \[
         V_r = \sup_{f \in \mathcal{F}} { (P-P_n)\frac{f}{\omega(f) + r}}.
         \] 
         We start with a modified version of the so-called peeling lemma:
          \begin{lem}
          
          Under the assumptions of Theorem 5.1, there exists a constant $C(\alpha)$ depending only on $\alpha$ such that, for all $r>0$,
          \[
          \mathbb{E}\left ( V_r \right ) \leq C(\alpha) \frac{\Phi(r)}{r}.
          \]
          Furthermore, we have $ C(\alpha) \underset{ \alpha \rightarrow 1}{\longrightarrow} \infty $.
          \end{lem}          
          \begin{proof}[Proof of Lemma 5.2]          
          Let $x > 1$ be a real number. Because $ 0 \in \mbox{Conv}(\mathcal{F}) $,
          \[
          \sup_{f \in \mathcal{F}} {(P-P_n) \frac{f}{\omega(f) + r}} \leq \sup_{\omega(f) \leq r}(P-P_n)\frac{f}{\omega(f) + r}  +   \sum_{k \geq 0}{\sup_{rx^k<\omega(f) \leq r x^{k+1}}} (P-P_n) \frac{f}{\omega(f) + r}.
          \]
          Taking expectation on both sides leads to
          \[
          \mathbb{E}(V_r) \leq \frac{\Phi(r)}{r} + \sum_{k \geq 0}{\frac{\Phi(rx^{k+1})}{r(1+x^k)}}.
          \] 
          Recalling that $\Phi$ is a sub-$\alpha$ function, we may write $ \Phi(r x^{k+1}) \leq x^{\alpha(k+1)}\Phi(r)$. Hence we get
           \[
           \begin{aligned}
           \mathbb{E}(V_r) & \leq \frac{\Phi(r)}{r} + \frac{\Phi(r)}{r} \sum_{k \geq 0} {\frac{x^{\alpha(k+1)}}{1 + x^k}} \\
                          & \leq \frac{\Phi(r)}{r} \left ( 1 + x^\alpha \left ( \frac{1}{2} + \frac{1}{x^{1-\alpha}-1} \right ) \right ).
                          \end{aligned}
                          \] 
                          Taking $C(\alpha) = \underset{x>1}{\inf} {\left ( 1 + x^\alpha \left ( \frac{1}{2} + \frac{1}{x^{1-\alpha}-1} \right ) \right )} $ proves the result.
                          
                          \end{proof}
                          
                          We are now in a position to prove Theorem 5.1. Using the inequality of Talagrand for a supremum of bounded variables that Bousquet \cite{Bousquet02} offered, we have, with probabilty larger than $ 1-e^{-x} $,
                          \[
                          V_r \leq \mathbb{E}(V_r) + \sqrt{\frac{x}{2rn}} + 2 \sqrt{\frac{xb \mathbb{E}(_r) }{nr}} + \frac{bx}{3rn}.
                          \] 
                          Using Lemma 5.2 and the inequality $ (a+b)^2 \leq 2(a^2 + b^2) $, 
                          \[
                           V_r \leq \frac{2 C(\alpha) \Phi(r)}{r} + \sqrt{\frac{x}{2rn}} + \frac{4}{3}\frac{bx}{rn}.
                           \]
                           Let $r^*$ be the solution of $ \Phi(r) = \frac{r}{D} $. If $r \geq r^*$, then $\frac{\Phi(r)}{r} \leq \left ( \frac{r^*}{r} \right ) ^{1-\alpha} \frac{1}{D}$. For such an $r$ we have 
                           \[
                           V_r \leq A_1 r ^{-(1-\alpha)} + A_2 r^{-1/2} + A_3 r^{-1},
                           \] 
                           with  
                           \[
                           \left \{
                           \begin{aligned}
                           & A_1 = \frac{ 2 C(\alpha) (r^*)^{1-\alpha}}{D} \\
                           & A_2 = \sqrt{\frac{x}{2n}} \\
                           & A_3 = \frac{4bx}{3n}
                           \end{aligned}
                           \right . .
                           \]
                           We want to find a suitable $r$ such that $r \geq r^*$ and $ V_r \leq 1/K $. To this aim, it suffices to see that if $r \geq ( 3K A_1)^\frac{1}{1-\alpha} + (3 K A_2)^2 + 3K A_3$, and $r \geq r^*$, then $V_r \leq 1/K$ using the previous upper bound on $V_r$.
                           
                           It remains to check that the condition $( 3K A_1)^\frac{1}{1-\alpha} + (3K A_2)^2 + 3K A_3 \geq r^* $ holds. To see this just recall that 
                           \[
                           ( 3K A_1)^\frac{1}{1-\alpha} = r^* \times \left ( \frac {6 K C(\alpha)}{D} \right )^\frac{1}{1-\alpha}.
                           \]
                           Thus, we deduce that, if $D \leq 6 K C(\alpha)$, the choice $r=( 3K A_1)^\frac{1}{1-\alpha} + (3K A_2)^2 + 3K A_3$ guarantees $ V_r \leq K^{-1} $ and, consequently, 
                           \[
                           Pf - P_n f \leq \frac{1}{K} \left ( \omega(f) + \left (\frac{6K C(\alpha)}{D} \right )^{\frac{1}{1-\alpha}}r^* + \frac{(9K^2+16Kb)x}{4n} \right ).
                           \]      
                           
                           \subsection{Proof of Proposition 5.1}
                           
               Using the differentiability of $P\gamma(\c,.)$, we get, for any $\c \in (\mathbb{R}^d)^k$ and $\c^* \in \mathcal{M}$,               
               \[
               \gamma(\mathbf{c},x) = \gamma(\mathbf{c}^*,x) + \left\langle \mathbf{c}-\mathbf{c}^*),\Delta(\mathbf{c}^*,x) \right\rangle + \|\mathbf{c}-\mathbf{c}^*\|R(\mathbf{c}^*,\mathbf{c}-\mathbf{c}^*,x),
               \]
               where, with use of Pollard's \cite{Pollard82} notation
               \[
                \left \{
               \begin{aligned}
               & \Delta(\mathbf{c}^*,x) = -2( (x-c^*_1)\mathbf{1}_{V^*_1},...,(x-c^*_k)\mathbf{1}_{V^*_k} ) \\
               & R(\mathbf{c}^*,\mathbf{c}-\mathbf{c}^*,x) = \sum_{i,j =1, \hdots, k} {\mathbf{1}_{V^*_i} \mathbf{1}_{V_j} \| \c - \c^* \|^{-1} \left [ 2 (c_i - c_j)^* x + \|c_i^*\|^2 - 2 (c_i^*)^* c_i + \|c_j\|^2 \right ]}
               \end{aligned}
               \right .  .
               \]
Observe that, because $\mathcal{M}$ is a finite set, by dominated convergence Theorem,
               \[
               R(\mathbf{c}^*,\mathbf{c}-\mathbf{c}^*,.) \overset{L_2}{\underset{ a.s}{\longrightarrow}} 0 \quad \mbox{when} \quad (\mathbf{c}-\mathbf{c}^*) \rightarrow 0.
               \]
               Splitting the expectation in two parts, we obtain

               \begin{align}
               \mathbb{E} \left ( \sup_{\c^* \in \mathcal{M}, \|\mathbf{c}-\mathbf{c}^*\|^2 \leq \delta} {(P_n-P)(\gamma(\mathbf{c}^*,.) - \gamma(\mathbf{c},.))}\right ) & \leq 
               \mathbb{E} \left (  \sup_{\c^* \in \mathcal{M}, \|\mathbf{c}-\mathbf{c}^*\|^2 \leq \delta} { (P_n-P) \left\langle -(\mathbf{c}-\mathbf{c}^*), \Delta(\mathbf{c}^*,.) \right\rangle  } \right ) \nonumber \\
               & \quad + \sqrt{\delta}  \mathbb{E} \left ( \sup_{\c^* \in \mathcal{M}, \|\mathbf{c}-\mathbf{c}^*\|^2 \leq \delta} {(P_n-P)(-R(\mathbf{c}^*,\mathbf{c}-\mathbf{c}^*,. ))} \right ) \label{complexity} \\
               &:= A + B . \nonumber             
               \end{align}

               \subsubsection{Term $A$: complexity of the model}

               Term A in inequality \eqref{complexity} is at first sight the dominant term in the expression $\Phi(\delta)$. The upper bound we obtain below is rather accurate, due to the finite-dimensional Euclidean space structure. Indeed, we have to bound a scalar product when the vectors are contained in a ball, thus it is easy to see that the largest value of the product matches in fact the largest value of the coordinates of the gradient term. We recall that $\mathcal{M}$ denotes the finite set of optimal vector of clusters. Let $\mathbf{x}=(x_1, \hdots, x_k)$ be a vector in $(\mathbb{R}^d)^k$. We denote by ${x_i}_r$ the $r$-th coordinate of $x_i$, and name it the $(i,r)$-th coordinate of $\mathbf{x}$. We may write
               \[
               \begin{aligned}
               \sup_{\mathbf{c}^* \in \mathcal{M},\|\mathbf{c}-\mathbf{c}^*\| \leq \sqrt{\delta}}{\left\langle \mathbf{c}-\mathbf{c}^*,(P_n - P)( - \Delta(\mathbf{c}^*)) \right\rangle } 
               & \leq 2\sup_{\mathbf{c}^* \in \mathcal{M}}{\sup_{j=1, \hdots, k, r=1, \hdots, d}{ \left | \frac{1}{n} \sum_{i=1}^{n} {(X_{i} - c_j^*)\mathbf{1}_{V_j^*}(X_i) } \right |_r}} \times \sqrt{\delta} \\
               & \leq \left\langle \mathbf{c}_{j,r,\varepsilon,\mathbf{c}^*} - \mathbf{c}^*, (P_n - P) (- \Delta (\mathbf{c}^*,.)) \right\rangle,
               \end{aligned}
               \]
        where $\mathbf{c}^*$ expresses the maximum and thus, have the largest possible coordinate absolute value for $(P_n - P) (- \Delta(\mathbf{c}^*)$)). Moreover, we denote by $(j,r)$ the coordinate of this largest possible absolute value, $\varepsilon$ the sign of its $(j,r)$-th coordinate, and $\mathbf{c}_{j,r,\varepsilon,\mathbf{c}^*} = \mathbf{c}^* + e_{j,r,\varepsilon}$, where $e_{j,r,\varepsilon}$ is the vector with $ \varepsilon \sqrt{\delta} $ for its $(j,r)$ coordinate, $0$ elsewhere. Therefore we can reduce the set of the $\mathbf{c}$'s of interest to a finite set, writing
               \[
                \sup_{\c^* \in \mathcal{M}, \|\mathbf{c}-\mathbf{c}^*\| \leq \sqrt{\delta}}{\left\langle \mathbf{c}-\mathbf{c}^*,(P_n - P) (- \Delta(\mathbf{c}^*)) \right\rangle } \leq
                \sup_{\mathbf{c}^* \in \mathcal{M}}{\sup_{j=1, \hdots, k,r=1, \hdots, d,\varepsilon = \pm 1} \left\langle \mathbf{c}_{j,r,\varepsilon,\mathbf{c}^*} - \mathbf{c}^*, (P_n - P) (-\Delta(\mathbf{c}^*,.)) \right\rangle}.
                \]
                Taking into account that for every $ \mathbf{c}^*$ in $\mathcal{M}$, $P \Delta(\mathbf{c}^*,.) =0$, and that for every fixed $\mathbf{c}$ and $\mathbf{c}^*$, the quantity $ \left\langle \mathbf{c} - \mathbf{c}^*, P_n (-\Delta(\mathbf{c}^*,.)) \right\rangle $ is a sub-Gaussian random variable with variance $ 16\delta/n $, we get, by a maximal inequality (Massart, \cite[Part 6.1]{Massart03}): 
                \[
                \mathbb{E} \left ( \sup_{\c^* \in \mathcal{M}, \|\mathbf{c}-\mathbf{c}^*\|^2 \leq \delta} { (P_n-P) \left\langle -(\mathbf{c}-\mathbf{c}^*), \Delta(\mathbf{c}^*,.) \right\rangle  } \right )
                \leq
                \left (4\sqrt{2\log(2\left| \mathcal{M} \right |)}\frac{\sqrt{kd}}{n} \right ) \times \sqrt{\delta}.
                \]
                Therefore, the expected dominant term involves the complexity of the model in a way which is proportional to the square root of the complexity. In our case, this complexity is the dimension of the vector of clusters space.

                \subsubsection{Bound on $B$}

                To bound the second term in inequality \eqref{complexity}, we follow the approach of Pollard \cite{Pollard82}, using complexity arguments such as Dudley's entropy integral.

                Let $\mathcal{F}$ be a set of functions defined on $\mathcal{X}$ with envelope $F$. Let $S$ be a finite set and $f$ a function. We denote $\|f\|_{l^2(S)} = \left (1/n \sum_{x \in S}{f^2(x)} \right )^{1/2}$, where $n = \mbox{Card}(S)$, and by $N_F ( \varepsilon,S,\mathcal{F} )$ the smallest integer $m$ such that there exists $\phi_1, \hdots, \phi_m$,  $m$ functions on $\mathcal{X}$ satisfying $ {\min_{i=1, \hdots, m}} { \| f - \phi_i \|_{l^2(S)}} \leq \varepsilon^2 \|F\|_{l^2(S)}^2$. Also define $H(\varepsilon) = {\sup_{S < \infty }} { \log N_F(\varepsilon,S,\mathcal{F})}$.

                According to \cite{Pollard82} and \cite[Theorem 7]{Pollard82a}, for the class of functions 
                \[
                 \mathcal{F} = \left \{ R(.,\mathbf{c}^*,\mathbf{c}-\mathbf{c}^*), \mathbf{c}^* \in \mathcal{M}, \| \mathbf{c}-\mathbf{c}^*\| \leq \sqrt{\delta} \right \}, 
                 \] 
                 there exists $C>0$ depending on $k$ and $d$ such  that $F(x) = C( 1+\| x\|)$ is an envelope for $\mathcal{F}$. Furthermore, for this envelope, we have
                \[
                H(\varepsilon) \leq \log(A) - W \log(\varepsilon),
                \]
                where $A$ is a positive constant, and $W$ depends only on the pseudo-dimension of $\mathcal{F}$, in a way which will not be described here (see the result of Pollard \cite[Theorem7]{Pollard82a}). We will use a classical chaining argument to bound term $B$. Let $\tilde{\c}$ denote the pair $(\c,\c^*) \in (\mathcal{B}(0,1))^k \times \mathcal{M}$. For practical, let $f_\mathbf{\tilde{\c}}$ denote the function $R(.,\mathbf{c}^*,\mathbf{c}-\mathbf{c}^*)$. We set $\varepsilon_0 = 1$ and $\varepsilon_j = 2^{-j} \varepsilon_0$.

                For any $f_{\tilde{\mathbf{c}}}$, let $f_{\tilde{\mathbf{c}}_j}$ be a function such that $\| f - f_{\tilde{{\mathbf{c}}}_j} \|^2_{l^2(S)} \leq \varepsilon_j^2 \| F \|^2_{l^2(S)}$, for every finite set $S$. Since Assumption 1 holds, $F$ is bounded from above by a constant $C_F$. By dominated convergence Theorem we have $ f_{\tilde{\mathbf{c}}_j} \underset{ j \rightarrow \infty }{\overset{L^1,a.s}{\longrightarrow}} f_{\tilde{\c}}$, and thus
                \[
                (P_n - P)f_{\tilde{\mathbf{c}}} = (P_n - P)f_{\tilde{{\mathbf{c}}_0}} + \sum_{j=1}^{\infty} (P_n - P)(f_{\tilde{\mathbf{c}}_j} - f_{\tilde{\mathbf{c}}_{j-1}}).
                \]
                Therefore
                \[
                \begin{aligned}
                \mathbb{E} \left ( \sup_{\mathbf{c}^* \in \mathcal{M}, \| \mathbf{c} - \mathbf{c}^* \| \leq \sqrt{\delta}}{ (P_n - P)f_{\tilde{\mathbf{c}}}} \right ) & \leq
                \mathbb{E} \left ( \sup_{\mathbf{c}^* \in \mathcal{M}, \| \mathbf{c} - \mathbf{c}^* \| \leq \sqrt{\delta}}{(P_n - P)f_{\tilde{\mathbf{c}}_0}} \right ) \\
                & \quad +  \sum_{j> 0}{ \mathbb{E} \left ( \sup_{\mathbf{c}^* \in \mathcal{M}, \| \mathbf{c} -\mathbf{c}^* \| \leq \sqrt{\delta}}{(P_n - P)(f_{\tilde{\mathbf{c}}_j} - f_{\tilde{\mathbf{c}}_{j-1}})} \right ) }.
                \end{aligned}
                \]
                Using a symmetrization inequality and introducing some Rademacher random variables $\sigma$ ($\sigma = \pm 1$ with probability $1/2$), we get, for the first term:

                \[
                \begin{aligned}
                \mathbb{E} \left ( \sup_{\mathbf{c}^* \in \mathcal{M}, \| \mathbf{c} - \mathbf{c}^* \| \leq \sqrt{\delta}}{(P_n - P)f_{\tilde{\mathbf{c}}_0}} \right ) & \leq 2 \mathbb{E}_{X} \mathbb{E}_{\sigma}  \left (\sup_{\mathbf{c}^* \in \mathcal{M}, \| \mathbf{c} - \mathbf{c}^* \| \leq \sqrt{\delta}}{ \frac{1}{n} \sum_{i=1}^{n}{\sigma_i f_{\tilde{\mathbf{c}}_0}(X_i)}} \right )\\
                & \leq 2\sqrt{2} \mathbb{E}_{X} \left ( \sqrt{\sup_{\c^*, \mathbf{c}}{ \|f_{\tilde{\mathbf{c}}_0}\|^2_{L^2(P_n)} \log(m(\varepsilon_0))}} \right ) \\
                & \leq 2\sqrt{2} \mathbb{E}_{X} \left ( \sqrt{\|F\|^2_{L^2(P_n)} \log(m(\varepsilon_0))} \right ) \\
                & \leq 2\sqrt{2} \mathbb{E}_{X} \left ( \sqrt{C_F^2 \log(m(\varepsilon_0))} \right ) \\
                & \leq \frac{\kappa_A}{\sqrt{n}},
                \end{aligned}
                \]
where $\kappa_{A}$ depends on $k$, $d$ and $P$. In the second line of this inequality, we used the maximal inequality for random processes depending only on Rademacher variables given by Massart \cite[part 6.1]{Massart03}. 
                It remains to bound the second term. Using the same approach (symmetrization and maximal inequality for Rademacher variables) we get, for every $j>0$, 
                \[
                \begin{aligned}
                \mathbb{E}\left( \sup_{\mathbf{c},\c^*} {(P_n-P)(f_{\tilde{\mathbf{c}}_j} - f_{\tilde{\mathbf{c}}_{j-1}})} \right ) & \leq
                2 \mathbb{E}_{X} \left ( \sqrt{ \frac{2}{n}\log (m(\varepsilon_j)m(\varepsilon_{j-1})) \sup_{\mathbf{c},\c^*}{\|f_{\tilde{\mathbf{c}}_j} - f_{\tilde{\mathbf{c}}_{j-1}}\|^2_{L^2(P_n)}}} \right ).
                \end{aligned}
                \]
                However $\|f_{\tilde{\mathbf{c}}_j} - f_{\tilde{\mathbf{c}}}\|_{L^2(P_n)} \leq  \varepsilon_j \|F\|_{L^2(P_n)} $, consequently 
                \[
                \begin{aligned}
                \|f_{\tilde{\mathbf{c}}_j} - f_{\tilde{\mathbf{c}}_{j-1}}\|^2_{L^2(P_n)} & \leq 4 \varepsilon_{j-1}^2 \|F\|^2_{L^2(P_n)} \\
                & \leq 4 C_F^2 \varepsilon_{j-1}^2.
                \end{aligned}
                \]
                Comparing a sum with an integral, we obtain 
                \[
                \sum_{j> 0}{ \mathbb{E} \left ( \sup_{\mathbf{c}^* \in \mathcal{M}, \| \mathbf{c} -\mathbf{c}^* \| \leq \sqrt{\delta}}{(P_n - P)(f_{\tilde{\mathbf{c}}_j} - f_{\tilde{\mathbf{c}}_{j-1}})} \right ) } \leq \frac{32}{\sqrt{n}} \int_{0}^{\varepsilon_1}{\sqrt{\log(m(\varepsilon))} d\varepsilon},
                \]
                which, by assumption on $m(\varepsilon)$, can be bounded from above by $\frac{\kappa_B}{\sqrt{n}}$, where $\kappa_B$ depends on $k$, $d$ and $P$.

                We are now in position to prove Proposition 5.1. From the two above subsections we deduce that
                \begin{center}
                \[
                \begin{aligned}
                \Phi(\delta) & = \mathbb{E} \left ( \sup_{\c^* \in \mathcal{M}, \|\mathbf{c}-\mathbf{c}^*\|^2 \leq A_2\delta} {(P_n-P)(\gamma(\mathbf{c}^*,.) - \gamma(\mathbf{c},.))}\right )\\
                             & \leq \sqrt{\delta} \frac{C}{\sqrt{n}}.
                             \end{aligned}
                             \]
                             \end{center}
     This concludes the proof.
                             
                             \subsection{Proof of Theorem 3.3}

                                    Let $\mathbf{x} = (x_1,\hdots,x_k)$ be a $k\times d$ vector, $V_1,\hdots,V_k$ the Voronoi diagram associated with an optimal vector of clusters $\mathbf{c}^*$. We state here a sufficient condition for the Hessian matrix $H(\mathbf{c}^*)$ to be positive. Denote $r_{i,j} = \|c^*_i - c^*_j\|$. It holds
                                    \[
                                    \left\langle H \mathbf{x},\mathbf{x} \right\rangle  = \sum_{i=1}^{k}{ \left [\left\langle H_{i,i} x_i,x_i \right\rangle  + \sum_{i\neq j}{\left\langle H_{i,j} x_j,x_i \right\rangle }\right ]},
                                    \]  
                                    where, for all $i=1, \hdots, k$, 
                                    \[
                                    \begin{aligned}
                                    \left\langle H_{i,i} x_i,x_i \right\rangle  + \sum_{i\neq j}{\left\langle H_{i,j} x_j,x_i \right\rangle } = \quad & 2P(V_i)\|x_i\|^2 - 2 x_i^t \left ( \sum_{j\neq i}{ r_{i,j}^{-1} \int_{\partial (V_i \cap V_j )} {f(u)(u-c^*_i)(u-c_i^*)^t du} } \right ) x_i \\
                                    & + 2 x_i^t \sum_{i\neq j}{r_{i,j}^{-1} \left ( \int_{\partial (V_i \cap V_j )}{f(u)(u-c_i^*)(u-c_j^*)^t du}\right ) x_j}.
                                    \end{aligned}
                                    \]
                                    The support of $P$ is included in $\mathcal{B}(0,1)$, thus we can replace $\partial ( V_i \cap V_j )$ with $\partial ( V_i \cap V_j ) \cap \mathcal{B}(0,1) $ in the equations above. However, to lighten notation we will omit the indication and implicitly assume that every set we consider is contained in $\mathcal{B}(0,1)$. Let $p_{i,j} = \int_{\partial(V_i \cap V_j)}{f(u) du}$ be the $d-1$-dimensional $P$-measure of the boundary between $V_i$ and $V_j$. Recalling that the underlying norm is the Euclidean norm, even for matrixes, we may write
                                    \[
                                    \begin{aligned}
                                    \left\langle H_{i,i} x_i,x_i \right\rangle  + \sum_{i\neq j}{\left\langle H_{i,j} x_j,x_i \right\rangle } \geq & \quad 2 P(V_i) \|x_i\|^2 - 2 \|x_i\|^2 \left \| \sum_{j\neq i}{ r_{i,j}^{-1}\int_{\partial (V_i \cap V_j )}{f(u)(u-c^*_i)(u-c_i^*)^t du} } \right \| \\
                     & - 2 \|x_i\| \left \| \sum_{j\neq i}{r_{i,j}^{-1} \left ( \int_{\partial (V_i \cap V_j )}{f(u)(u-c_i^*)(u-c_j^*)^t du} \right ) x_j}  \right \|,
                                     \end{aligned}
                                     \]
                                     with 
                                     \[
                                     \begin{aligned}
                                     \left \| \sum_{i\neq j}{r_{i,j}^{-1}  \left( \int_{\partial (V_i \cap V_j )}{f(u)(u-c_i^*)(u-c_j^*)^t du} \right ) x_j}  \right \| & \leq \sum_{j\neq i}{r_{i,j}^{-1} \left \| \left ( \int_{\partial (V_i \cap V_j )}{f(u)(u-c_i^*)(u-c_j^*)^t du} \right ) x_j \right \| }\\
                                                        & \leq \sum_{j\neq i}{r_{i,j}^{-1} \left( \int_{\partial (V_i \cap V_j )}{f(u)\|u-c_i^*\|\|u-c_j^*\| du} \right ) \|x_j\|} \\ 
                                                        & \leq \sum_{j\neq i} {r_{i,j}^{-1} p_{i,j} 4\|x_j\|}.
                                     \end{aligned}
                                     \]
                                      Next,
                                     \[
                                     \begin{aligned}
                                     \left\langle H_{i,i} x_i,x_i \right\rangle  + \sum_{i\neq j}{\left\langle H_{i,j} x_j,x_i \right\rangle } & \geq \left ( 2 P(V_i) - \frac{8}{B}  \sum_{i\neq j}{p_{i,j}}  \right ) \|x_i\|^2 \\
                & \quad - \left ( \frac{8}{B}\sum_{i \neq j}{p_{i,j}} \right ) \|x_i\| \|x_j\|,
                                     \end{aligned}
                                     \]
                                     where we recall that $\|x_i\| \|x_j\| \leq 2 \left ( \|x_i\|^2 + \|x_j\|^2\right )$ and $B = \underset{i\neq j,\mathbf{c}^* \in \mathcal{M}}{\inf} {\|c_i^* - c_j^*\|}$. Summing with respect to $i$ leads to
                                     \begin{center}
                                     \[
                                     \begin{aligned}
                                     \left\langle H \mathbf{x}, \mathbf{x} \right\rangle \geq & \sum_{i=1}^{k}{\left (2P(V_i) - \frac{40 M^2}{B}\sum_{j \neq i}{p_{i,j}} \right ) \|x_i\|^2}. 
                                     \end{aligned}
                                     \]
                                     \end{center}
                                     The last step is to derive bounds for $p_{i,j}$ from conditions on $f$. Denote $\lambda = \|f\|_{\infty}$, we see that
                                     \[
                                     \sum_{j\neq i}{ p_{i,j}} = \int_{\partial V_i}{f(u) du}.
                                     \]                                   
                                     $ V_i$ is a regular convex set included in $\mathcal{B}(c_i^*,2)$. Therefore, by a direct application of Stokes Theorem, the surface of $\partial V_i$ is smaller than the surface of $\mathcal{S}_{d-1}(c_i^*, 2)$ (the sphere of radius $2$). Consequently
                                     \[
                                     \sum_{j\neq i}{p_{i,j}} \leq \lambda \frac{2 \pi^{d/2}}{\Gamma(d/2)} (2)^{d-1}.
                                     \]
                                     It follows that $\lambda \leq \frac{B \Gamma(d/2)}{2^{d+5}\pi^{d/2}} \underset{i=1, \hdots, k}{\inf} {P(V_i)}$ is enough to ensure that the Hessian matrix $H(\c^*)$ is positive definite. 
                            
                            \subsection{Proof of Proposition 4.4}

                             We take a distribution uniformly distributed over small balls far one from each others. Denote by $V_i$ the Voronoi cell associated with $z_i$ in $(z_1,\hdots,z_k)$. Let $Q$ be a $k$-quantizer, $Q^*$ the expected optimal quantizer which maps $V_i$ to $z_i$ for all $i$. Denote finally, for all $i=1, \hdots, k$, $R_i(Q) = \int_{V_i}{\|x - Q(x)\|^2dx}$ the contribution of the $i$-th Voronoi cell to the risk of $Q$. 
                             
                             First we compute 
                             \[
                             \begin{aligned}
                             R_i(Q^*) & = \frac{1}{k \rho^d V} \int_{0}^{\rho}{S r^{d+1}dr} \\
                                      &  = \frac{\rho^2d}{k(d+2)},
                             \end{aligned}
                             \]
                             where $S$ and $V$ are the unit surface and the volume of the unit ball in $\mathbb{R}^d$.

                             Let $i$ be an integer between $1$ and $k$. Let $m_i^{in} = \left |Q(\mathcal{B}_d(z_i,\rho)) \cap V_i \right |$ be the number of images of $V_i$ sent by $Q$ inside $V_i$, and let $m_i^{out} = \left |Q(\mathcal{B}_d(z_i,\rho)) \cap V_i^c \right |$ be the number of images of $V_i$ sent outside $V_i$. The three situations of interest are the following ones:
                             \begin{itemize}
                             \item[$\rightarrow$]If $m_i^{in} = 1$ and $m_i^{out}=0$, it is clear that $R_i(Q) \geq R_i(Q^*)$.
                             \item[$\rightarrow$]If $m_i^{in} \geq 2$ and $m_i^{out} =0$, then we just can see that $R_i(Q) \geq R_i(Q^*) - \frac{\rho^2 d }{k(d+2)} = 0$.
                             \item[$\rightarrow$] At last, suppose that $m_i^{out} \geq 1$. Then there exist $x \in \mathcal{B}_d(z_i,\rho) $ such that 
                             \[
                             \left \{
                             \begin{aligned}
                              \|Q(x)-x\| & \leq \underset{c \in Q(\mathcal{B}_d(z_i,\rho))}{\inf} \|x-c\| \\
                              \|Q(x)-x\| & \geq d(z_i, V_i^c) - \rho \geq \frac{R}{2} - \rho 
                              \end{aligned}
                              \right .  .
                              \]
                              Let $c\in Q(\mathcal{B}_d(z_i,\rho))$. Then
                                    \[
                                    \begin{aligned}
                                    \|c-z_i\| & \geq \|c-x\| - \rho \\
                                              & \geq \|Q(x) - x \| - \rho \\
                                              & \geq \frac{R}{2} - 2\rho .
                                     \end{aligned}
                                     \] 
                                     Then, we deduce that, for every $y \in \mathcal{B}_d(z_i,\rho)$ and $c \in Q(\mathcal{B}_d(z_i,\rho))$, $\|y-c\| \geq  \frac{r}{2} - 3\rho$. Therefore
                                     \[
                                     \begin{aligned}
                                     R_i(Q) & \geq \frac{\left ( \frac{R}{2} - 3 \rho \right )^2}{k}\\
                                            & \geq R_i(Q^*) + \frac{1}{k} \left ( \left ( \frac{R}{2} - 3 \rho \right )^2 - \frac{\rho^2 d }{d+2} \right ).
                                     \end{aligned}
                                     \]
                                     \end{itemize}
                                     Now suppose that $m_i^{in} \geq 2$. Then at least two clusters of $Q$ lies in $V_i$. Therefore, there exists $j$ such that no cluster of $Q$ lies in $V_j$, so that $m_j^{out} \geq 1$. We  straightforward deduce that the number of cells $V_i$ for which $m_i^{in} \geq 2$ is smaller than the number of cells for which $m_j^{out} \geq 1$.

                                    Taking into account all contributions of Voronoi cells, we get
                                    \begin{center}
                                    \[
                                    \begin{aligned}
                                    P\gamma(c,.) = R(Q) & = \sum_{\{i;m_i^{in} \geq 2,m_i^{out} =0 \}}{R_i(Q)} + \sum_{\{i;m_i^{out} \geq 1\}}{R_i(Q)} + \sum_{\{i;m_i^{in} =1,m_i^{out} =0\}}{R_i(Q)} \\
                                                        & \geq R(Q^*) + \sum_{\{i;m_i^{in} \geq 2,m_i^{out} =0\}}{\frac{1}{k} \left ( \left ( \frac{R}{2} - 3 \rho \right )^2 - \frac{2\rho^2d}{d+2} \right )},
                                    \end{aligned}
                                    \]
                                    \end{center}
                                    from which we deduce a sufficient condition to get $R(Q) \geq R(Q^*)$.

                                    \subsection{Proof of Proposition 4.3}
                                    Using the same method as in the proof of Proposition 4.4, we prove that, for $\eta = 2$ and $R=10$, the optimal vector of clusters is $(\frac{\eta}{2},\frac{\eta}{2} + R, \frac{\eta}{2} + 2R)$. Thus the density $f$ is zero-valued on each boundary of every Voronoi cell of the optimal vector of centroids. Consequently Assumption 2 is satisfied. For $n$ large enough,

                                    \[
                                    \begin{aligned}
                                    P\gamma(\mathbf{c}_n,.) = & \frac{1}{3\eta}\int_{0}^{\eta}{x^2dx} + \frac{1}{3\eta}\int_{R}^{R + \eta}{x^2dx} + q \int_{\eta + 2R -1}^{\frac{n}{2}}{x^2 e^{-x} dx}\\
                                    & + q \int_{\frac{n}{2}}^{\frac{n^2 + n}{2}}{(x-n)^2 e^{-x}dx} + q \int_{\frac{n^2 + n}{2}}^{+ \infty}{ \left ( x - n^2 \right )^2 e^{-x}dx}.
                                    \end{aligned}
                                    \]
                                    Hence, by the dominated convergence Theorem for the three first terms of the right-hand side and through computation for the remaining terms, $P\gamma(\mathbf{c}_n,.) \underset{n \longrightarrow \infty}{\longrightarrow} P\|x\|^2$.
                                    \subsection{Proof of Proposition 4.5}
                  
                  We begin with a lemma which ensures that every possible optimal centroid $c_i^*$ is close to at least one mean $m_j$ of the mixture when the ration $p_{min}/p_{max}$ is large enough.
                  
                  \begin{lem}
                   Let $\c^*$ be an optimal vector of clusters. Suppose that 
                   \[
                   \frac{p_{min}}{p_{max}} \geq \frac{288 k \sigma^2}{(1-\varepsilon)\tilde{B}^2(1-e^{-\tilde{B}^2/288\sigma^2})}.
                   \]
                   Then, for every $j=1, \hdots, k$ there exists $i\in \{1, \hdots k\}$ such that $\|m_j - c^*_i\| \leq \frac{\tilde{B}}{6}$.
                   \end{lem}                
                 \begin{proof}[Proof of Lemma 5.3]
                 
                 Denote by $\m$ the vector of clusters $(m_1, \hdots, m_k)$, and by $M_i$ the Voronoi cell associated with $m_i$. We bound from above the quantity $P\gamma(\m,.)$:
                 \[
                 \begin{aligned}
                 P\gamma(\m,.) & = \sum_{i=1}^{k}{\frac{p_i}{2 \pi \sigma^2N_i} \int_{M_i}{\| x-m_i\|^2e^{-\|x-m_i\|^2/2\sigma^2}}dx} \\
                               & \leq \sum_{i=1}^{k}{\frac{p_i}{2 \pi \sigma^2N_i} \int_{\mathbb{R}^2}{\| x-m_i\|^2e^{-\|x-m_i\|^2/2\sigma^2}}dx}\\
                               & \leq \frac{2kp_{max} \sigma^2}{(1-\varepsilon)}.
                 \end{aligned}
                 \]
                 Let $\c$ be a vector of clusters such that there exists $j$ satisfying, for all $i=1, \hdots, k$, $\|m_j-c_i\| > \tilde{B}/6$. We will prove that $P\gamma(\c,.) > P\gamma(\m,.)$, which implies that $\c \notin \mathcal{M}$.
                 In fact we have, for all $i=1, \hdots, k$ and for all $x \in \mathcal{B}(m_j, \tilde{B}/12)$, $\|x-c_i\| > \tilde{B}/12$. Hence, a lower bound for $P\gamma(\c,.)$ is
                 \[
                 \begin{aligned}
                 P\gamma(\c,.) & \geq \int_{\mathcal{B}(m_j,\tilde{B}/12)}{\min_{i=1, \hdots, k}{\|x-c_i\|^2}f(x)dx} \\
                 & > \frac{\tilde{B}^2}{144} \sum_{i=1}^{k}{\frac{p_i}{2 \pi \sigma^2 N_i} \int_{\mathcal{B}(m_j, \tilde{B}/12)}{e^{-\|x-m_i\|^2/2\sigma^2}dx}} \\
                 & > \frac{\tilde{B}^2 p_j}{288 \pi \sigma^2 N_j} \int_{\mathcal{B}(m_j, \tilde{B}/12)}{e^{- \|x-m_j\|^2/2\sigma^2}dx} \\
                 & > \frac{p_{min}\tilde{B}^2}{144}\left ( 1 - e^{-\tilde{B}^2/288\sigma^2} \right )\\
                 & > P\gamma(\m,.).
                 \end{aligned}
                 \]
                 Hence we deduce that every optimal vector of clusters has a centroid close to every mean $m_j$ of the mixture, of at most $\tilde{B}/6$. 
                 \end{proof} 
                 
                 Suppose that the ratio $p_{min}/p_{max}$ satisfies the assumption of Proposition 4.5. In particular $p_{min}/p_{max}$ satisfies the assumption of Lemma 5.3. Then we deduce that, up to a reindexation, for every $\c^* \in \mathcal{M}$, $\|c_i^* - m_i\| \leq \tilde{B}/6$. We conclude that $2\tilde{B}/3 \leq B \leq 4 \tilde{B}/3$.
                 
                 Since, for all $i =1, \hdots, k$, $\mathcal{B}(c_i^*,B/2) \subset V_i^*$, it is easy to see that $\mathcal{B}(m_i,B/4) \subset \mathcal{B}(c_i^*,B/2) \subset V_i^*$, which leads to $N^* \subset \left ( \underset{i=1}{\overset{k}{\bigcup}} \mathcal{B}(m_i,B/4) \right )^c$. Consequently, in order to apply Theorem 3.2, we just have to prove that
                 \[
                 \|f_{\left | \left (\underset{i=1}{\overset{k}{\bigcup}}  \mathcal{B}(m_i,B/4) \right )^c \right .}\|_{\infty}  \leq  
 \frac{\Gamma\left (1\right ) B }{2^{7} \pi} \quad \underset{i=1, \hdots, k}{\inf} P\left ( \mathcal{B}(m_i,B/4) \right ).
                  \]
                 First we derive a lower bound for the right-hand side. For every $i=1, \hdots, k$, 
                 \[
                 \begin{aligned}
                 P(\mathcal{B}(m_i,B/4)) & \geq \frac{p_i}{N_i} \frac{1}{2\pi \sigma^2} \int_{\mathcal{B}_2(0,B/4)}{e^{-\frac{\|x\|^2}{2\sigma^2}}dx}  \\
                                        & \geq \frac{p_i}{N_i} \frac{1}{2 \pi \sigma^2} \times 2 \pi \int_{0}^{B/4}{r e^{-\frac{r^2}{2\sigma^2}}dr} \\
                                        & \geq p_{min}\left ( 1 - e^{-\frac{B^2}{32\sigma^2}} \right ).
                                        \end{aligned}
                                        \]                                   
     Then, we deal with the left-hand side. Let $x$ be at distance from every $m_i$ of at least $B/4$. Then
     \[
     \begin{aligned}
     f(x) & \leq \sum_{i=1}^{k} {\frac{p_i}{N_i} \frac{1}{2 \pi \sigma^2} e^{-^\frac{B^2}{32 \sigma^2}}} \\
          & \leq \frac{k p_{max}}{2 \pi \sigma^2(1-\varepsilon)} e^{-\frac{B^2}{32\sigma^2}}.
          \end{aligned}
          \]
                    The rest of the proof follows from straightforward computation, using the assumption of Proposition 4.5 and the relationship between $B$ and $\tilde{B}$: $2 \tilde{B}/3 \leq B \leq 4 \tilde{B}/3$. 
                    
\textbf{Remark} A careful reader should have noticed that the $k$ factor is suboptimal in the previous inequality. In fact we are able in this case to bound from above $f(x)$ with $\frac{1}{2 \pi \sigma^2(1-\varepsilon)} e^{-\frac{B^2}{32\sigma^2}}$. However, this bound does not involve $p_{max}$, and so involve a condition not on the ratio of extremal proportions of the mixture, but rather on the minimal proportion of the mixture, which is less natural. Moreover, the $p_{max}$-free bound is valid only in the equal variance case, (ie), when the variance $\sigma^2_i$ of any element of the mixture is the same. In general it is not the case and a condition as in  Proposition 4.5 for that kind of mixture would naturally involve the ratio $p_{min}/p_{max}$.

         \bibliographystyle{plain}
         \bibliography{biblio}
         \end{document}